%% file: main.tex
\documentclass[10pt]{article}
\usepackage[
  margin=3cm,
  includefoot,
  footskip=30pt,
]{geometry}

\usepackage[T1]{fontenc}
\usepackage{textcomp}

\usepackage{libertinus}

\usepackage{url}
\usepackage{fullpage}
\usepackage{diagbox}
\usepackage{array}
\usepackage[table]{xcolor}
\usepackage{arydshln}
\usepackage{natbib}

\usepackage{amsthm,amsfonts,amsmath,amssymb,epsfig,color,float,graphicx,verbatim,bm,bbm}
\usepackage{enumerate}
\usepackage{enumitem}
\usepackage{wrapfig}
\usepackage{subcaption}
\usepackage[export]{adjustbox}
\usepackage{nicefrac}
\usepackage{hhline}
\usepackage{multicol}
\usepackage{multirow}
\usepackage[dvipsnames]{xcolor}

\usepackage[ruled, vlined, linesnumbered]{algorithm2e}

\usepackage{hyperref}
\hypersetup{
	colorlinks   = true, 
	urlcolor     = blue!75!black, 
	linkcolor    = blue!75!black, 
	citecolor   = blue!75!black 
}

\input{symbols}

\title{
The Complexity of Finding Stationary Points\\in Nonsmooth Nonconvex Optimization
}
\author{Guy Kornowski}
\date{}

\begin{document}

\maketitle

\begin{abstract}
We prove that first-order algorithms require $\Omega(\delta^{-1}\eps^{-3})$ gradient queries (in the worst case) to find a $(\delta,\eps)$-Goldstein stationary point of a Lipschitz function, at which there is a convex combination of gradients within distance $\delta$ whose norm is at most $\eps$.
This lower bound is 
tight, matching known algorithms up to absolute constants, therefore resolving the complexity of convergence to stationarity in nonsmooth nonconvex optimization.
We further prove a tight lower bound of $\Omega(\lambda^{1/2}\eps^{-7/2})$ for finding points satisfying the recently proposed relaxed notion of $(\lambda,\eps)$-stationarity, which allows combining further-away gradients.
Our results reveal that convergence rates to nonsmooth stationarity are not affected by gradient stochasticity,
in sharp contrast to smooth optimization.
\end{abstract}

\section{Introduction}

We consider optimization problems associated with an objective function
$f:\reals^d\to\reals$ which is not necessarily smooth or convex. Such objectives are extremely common in modern applications, such as those involving neural networks,
but they cannot, in general, be globally minimized efficiently
\citep{Nemirovski-1983-Problem,Murty-1987-Some}. This discrepancy has motivated the introduction of local stationarity measures for analyzing the convergence of optimization methods. For example, for smooth objectives where $\nabla f$ is Lipschitz continuous, it is known that gradient descent can find an $\eps$-stationary point,
namely $\bx$ such that $\|\nabla f(\bx)\|\leq\eps$, within $O(1/\eps^{2})$ steps, independent of the dimension $d$. Moreover, this convergence rate is known to be optimal among all first-order methods \citep{carmon2020lower}.
Extending stationarity-based convergence analyses to nonsmooth objectives is not so simple.
Indeed, classical works only established asymptotic convergence in such settings \citep{Burke-2005-Robust,Kiwiel-2007-Convergence},
and while it remains true that even simple subgradient methods asymptotically converge to stationary points under mild assumptions \citep{davis2020stochastic}, getting anywhere near a point with small (sub-)gradient norm may require $\exp(\Omega(d))$ many steps \citep{kornowski2022oracle}.

In a landmark paper, \citep{zhang2020complexity} bypassed such hardness results and introduced a relaxed stationarity notion based on the Goldstein subdifferential \citep{goldstein1977optimization}. Roughly, a point $\bx$ is called $(\delta,\eps)$-Goldstein stationary if there exists a convex combination of gradients in a $\delta$-ball around $\bx$ whose norm is less than $\eps$. Remarkably, \citet{zhang2020complexity} showed that $O(1/\delta\eps^{3})$ first-order queries suffice to find such a point, independently of the dimension $d$.
Since then, Goldstein stationarity has found many applications and has been extensively studied, and there are by now many algorithms which are proven to converge efficiently to such points in various settings, including general Lipschitz optimization \citep{Tian-2022-Finite,davis2022gradient} 
and its variants such as stochastic \citep{cutkosky2023optimal}, zero-order \citep{lin2022gradient,chen2023faster,kornowski2024algorithm}, constrained \citep{grimmer2024goldstein,liu2024zeroth}, private \citep{zhang2023private,kornowski25improved}, and minimax optimization \citep{shi2026nonsmooth},
as well as higher-order extensions \citep{guan2026computing,guan2026hardness}.

A fundamental question throughout this vast body of work is understanding the complexity of finding $(\delta,\eps)$-Goldstein stationary points. For stochastic problems, \citet{cutkosky2023optimal} proved that  $\Theta(1/\delta\eps^3)$ stochastic gradient queries are both sufficient and necessary, thus matching the \emph{smooth} stochastic rate of $\Theta(1/\eps^{4})$ under the natural parameter regime $\delta\approx \eps$ \citep{ghadimi2013stochastic,arjevani2023lower}.
Interestingly though, for deterministic problems in which exact gradients are available, despite substantial effort, no algorithm is known to achieve a complexity better than $O(1/\delta\eps^3)$, attained already by \citep{zhang2020complexity}. In comparison, as previously discussed, deterministic smooth optimization requires only $O(1/\eps^{2})$ gradient queries.
This left the following fundamental question open:

\begin{center}
\begin{minipage}{0.9\linewidth}
\textit{
What is the gradient-oracle complexity of finding $(\delta,\eps)$-Goldstein stationary points of nonsmooth nonconvex functions? In particular,
do deterministic gradients enable faster convergence than stochastic ones?
} 
\end{minipage}
\end{center}

In this work, we resolve this question by proving that $\Omega(1/\delta\eps^{3})$ gradient queries are necessary to find a $(\delta,\eps)$-Goldstein stationary point, even when the gradients are deterministic. In light of previously discussed algorithmic results, this lower bound is tight.
Interestingly, our result implies that under the Goldstein framework, nonsmooth nonconvex optimization problems do not suffer from stochasticity,
thus answering an open question asked by \citet{cutkosky2023optimal}. 
In Table~\ref{tab: complexity}, we summarize the optimal (i.e., minimax) oracle complexities of stationarity in nonconvex optimization, with our result indicated in blue.

In addition, we consider the complexity of finding $(\lambda,\eps)$-stationary points, a stationarity notion proposed by \citet{zhang2024randomscaling}, which relaxes Goldstein stationarity by allowing gradients to be combined beyond a ball (the formal definition appears in Section~\ref{sec: prelim}). This relaxation has proved to be widely useful in analyzing various algorithms which are popular in practice, such as Adam \citep{ahn2024adam}, schedule-free methods \citep{ahn2025general},
and spectral optimizers \citep{jiang2026adaptive,li2026muon}.

In this context, we prove a lower bound showing that $\Omega(\lambda^{1/2}/\eps^{7/2})$ deterministic gradient queries are necessary to find a $(\lambda,\eps)$-stationary point. This tightly matches previously known algorithmic upper bounds up to constants \citep{zhang2024randomscaling}. Moreover, similarly to our conclusion for Goldstein stationarity, by matching even the optimal stochastic rate,
our result proves that gradient stochasticity does not affect the convergence rate to such points.

\begin{table}[h]
\centering

\renewcommand{\arraystretch}{1.5}
\setlength{\tabcolsep}{8pt}

\begin{tabular}{|c|c c|}
\hhline{|-|-|-|}
function~~\textbackslash~~gradient oracle
&
Deterministic
&
\multicolumn{1}{|c|}{Stochastic}
\\
\hline
Nonsmooth
&
${\color{blue}\Theta(\delta^{-1}\epsilon^{-3})}$
&
\multicolumn{1}{:c|}{$\Theta(\delta^{-1}\eps^{-3})$}
\\
\cline{1-2}\cdashline{3-3}Smooth
&
$\Theta(\eps^{-2})$
&
\multicolumn{1}{|c|}{$\Theta(\eps^{-4})$}
\\\hline
\end{tabular}

\caption{First-order complexities for finding stationary points of nonconvex functions (omitting problem parameters).
Notably, for $\delta\approx\eps$ all complexities coincide, except for smooth deterministic optimization.}
\label{tab: complexity}

\end{table}

\section{Preliminaries} \label{sec: prelim}

\paragraph{Notation.}

Let $[n]:=\{1,\dots,n\}$.
We use boldface font to denote vectors, e.g., let $\bx\in\reals^d$, whose $i^\textnormal{th}$ coordinate is denoted by $x_i$.
We denote $\bx\vleq\by$ if $x_i\leq y_i$ for all $i$.
We denote a vector's support by
$\supp(\bx):=\{i:x_i\neq 0\}$, and of a set by $\supp(\Xcal):=\cup_{\bx\in\Xcal}\supp(\bx)$.
Given a vector-valued function $\mathbf{f}:\reals^d\to\reals^n$ we denote $\mathbf{f}(\bx)=(f_1(\bx),\dots,f_n(\bx))$.
$\norm{\bx}$ denotes the Euclidean norm,
$\inner{\,\cdot\,,\,\cdot\,}$ is the standard Euclidean dot product,
and $\BB(\bx,\delta):=\{\by\in\reals^d:\norm{\by-\bx}\leq\delta\}$ denotes a closed ball. We denote by $\bzero$ the zero vector and by $\bone$ the all-ones vector, where the dimension is clear from context, and let $\dist(\bzero,\Xcal):=\inf_{\bx\in \Xcal}\|\bx\|$. We denote by $\S^{d-1}\subset\reals^d$ the unit sphere, and $\e_i$ is the $i^\textnormal{th}$ standard basis vector.
We denote the ReLU operation as $[x]_{+}=\max\{0,x\}$, applied to vectors coordinate-wise.
We use the standard big-O notation, with $O(\cdot)$, $\Theta(\cdot)$ and $\Omega(\cdot)$ hiding absolute constants that do not depend on problem parameters.

\paragraph{Nonsmooth Optimization.}

A function $f:\reals^d\to\reals$ is called $L$-Lipschitz if for any $\x,\y\in\reals^d:|f(\x)-f(\y)|\leq L\norm{\x-\y}$.
By Rademacher's theorem, Lipschitz functions are differentiable almost everywhere (with respect to Lebesgue measure). Hence, for any Lipschitz function $f:\reals^d\to\reals$ and point $\x\in\reals^d$ the Clarke subdifferential \citep{Clarke-1990-Optimization} can be defined as
\[
\partial f(\x):=\conv\{\g\,:\,\g=\lim_{n\to\infty}\nabla f(\x_n),\,\x_n\to \x\}~,
\]
namely, the convex hull of all limit points of $\nabla f(\x_n)$ over all sequences of differentiable points which converge to $\x$.
Given $\delta\geq 0$ the Goldstein $\delta$-subdifferential \citep{goldstein1977optimization} of $f$ at $\x$ is the set
\[
\partial_{\delta}f(\x):=\conv\left(\cup_{\y\in \B(\x,\delta)}\partial f(\y)\right)~,
\]
namely, all convex combinations of Clarke subgradients at points in a $\delta$-neighborhood of $\x$.

\begin{definition}[\citealp{zhang2020complexity}]
A point $\bx$ is a $(\delta,\epsilon)$-Goldstein stationary point of $f$
if $\dist(\bzero,\partial_\delta f(\bx))\leq\eps$.
\end{definition}

Given $\lambda>0$, we further denote
\[
\mathcal{P}_\lambda f(\bx):=\inf_{\substack{\text{random vector}~\by\in\reals^d\\\bg\in\partial f(\by)\text{~a.s.}}}\left\{\|\E[\bg]\|+\lambda\cdot\E\|\bx-\by\|^2\right\}
~.
\]

\begin{definition}[\citealp{zhang2024randomscaling}]
A point $\bx$ is a $(\lambda,\epsilon)$-stationary point of $f$
if $\Pcal_\lambda f(\bx)\leq\eps$.
\end{definition}

Note that by only considering distributions over $\by$
supported on $\BB(\bx,\delta)$, it is clear that
$\Pcal_\lambda f(\bx)\leq\dist(\bzero,\partial_\delta f(\bx))+\lambda\delta^2$, and so any $(\sqrt{\eps/2\lambda},\eps/2)$-Goldstein stationary point is $(\lambda,\eps)$-stationary.

\begin{remark}
\citet{zhang2024randomscaling} included an additional requirement in the definition of $\Pcal_\lambda$ that $\E[\by]=\bx$, which we chose not to include for a more direct comparison to Goldstein stationarity. Note that omitting this condition makes the definition easier to satisfy, hence only implies a stronger lower bound.
\end{remark}

\paragraph{Algorithms.}
We consider iterative first-order algorithms, in the standard oracle complexity framework \citep{Nemirovski-1983-Problem}:
At each time step $t\in\NN$, an algorithm produces an iterate $\bx_t$, receives the oracle response $(f(\bx_t),\partial f(\bx_t))$, and chooses its next iterate, possibly at random, based on all the previous oracle responses.\footnote{For the sake of proving a lower bound, assuming an algorithm can access the entire Clarke subdifferential, even though this is typically not the case, only makes the result stronger.}
In this work, we focus on the class of \emph{zero-respecting} algorithms \citep{carmon2020lower},
which is a general class defined as satisfying $\bx_0=\bzero$, and
\[
\Pr\left[\supp(\bx_t)\subseteq\cup_{t'<t}\supp(\partial f(\bx_{t'}))\right]=1
~~~~\text{for all~~}t\in\NN~.
\]
We note that this algorithm class subsumes the classical linear-span assumption \citep{Nesterov-2018-Lectures}, and includes essentially all first-order algorithms of interest in our setting.\footnote{Note that some algorithms analyzed in prior works perturb the iterates, which seemingly defies the zero-respecting condition. However, perturbations are only considered for the sake of ensuring that iterates are differentiable, which helps avoid computationally infeasible subgradient selection procedures. Since we allow the algorithm to see the entire subdifferential and do not assume bounded computation, no perturbations are needed for prior algorithms to work.} We further emphasize that this definition allows algorithms to be randomized, which is crucial in this context, since it is known that deterministic algorithms cannot converge to Goldstein stationary points in a dimension-free manner \citep{jordan2023deterministic}.

\section{Main Results}

We are now ready to present our lower bound for finding Goldstein stationary points:

\begin{theorem} \label{thm: main}
For some absolute constant $c>0$ the following holds:
For any $L,\Delta,\delta,\eps>0$ satisfying $\eps\leq c\cdot\min\{L,\delta L^2\Delta^{-1},\Delta\delta^{-1}\}$, there exists an $L$-Lipschitz function $f:\reals^d\to\reals$, where $d=\Theta({\Delta L^2}{\delta^{-1}\eps^{-3}})$,
satisfying $f(\bzero)-\inf f\leq\Delta$, such that for any zero-respecting first-order algorithm, all of its first $T$ iterates are not $(\delta,\eps)$-Goldstein stationary points of $f$, unless
\[
T\geq c\cdot \frac{\Delta L^2}{\delta\eps^3}~.
\]
\end{theorem}

As previously discussed, Theorem~\ref{thm: main}
is tight up to a numerical constant due to existing upper bounds, thus establishing the optimal first-order oracle complexity of finding Goldstein stationary points.
Moreover, the rate is matched by stochastic algorithms \citep{cutkosky2023optimal}, therefore establishing that the first-order oracle complexity of finding Goldstein stationary points does not depend on the stochasticity of the gradients.
In particular, this result improves the previously best-known lower bound for this task, which was $\Omega(1/\eps^2)$, due to \citet{kornowski2022complexity}. 

We will now sketch the idea behind the proof, which appears in Section~\ref{sec: proofs}.
At a high level, the proof follows a standard lower bound recipe \citep{Nemirovski-1983-Problem,Nesterov-2018-Lectures,carmon2020lower,arjevani2023lower}: construct a function $f$ which simultaneously satisfies $(i)$ a ``zero-chain'' property, meaning that the support of any subgradient at $\bx$ is not much larger than the support of $\bx$, and $(ii)$ all solution points (in this case, Goldstein stationary) have large ``progress'', meaning that their last coordinates are non-zero. Clearly, $(i)+(ii)$ imply a lower bound via the zero-respecting property.

As the starting point of our concrete construction, we first recall existing \emph{upper} bounds.
All prior algorithms for finding Goldstein stationary points follow (either explicitly or implicitly) Goldstein's conceptual method \citep{goldstein1977optimization}, which consists of a double-loop structure, with an outer iteration decreasing the function value by order $\delta\eps$, and an inner loop which finds such a descent direction by combining $O(1/\eps^2)$ gradients, thus leading to a complexity of $O(\frac{1}{\delta\eps}\cdot\frac{1}{\eps^2})=O(\frac{1}{\delta\eps^3})$ (assuming a bound on initial suboptimality).
Interestingly, the construction of our lower bound is directly informed by such upper bounds:
We partition coordinates into $B=\Omega(1/\delta\eps)$ ``blocks'' of length $n=\Omega(1/\eps^2)$ each, and construct a function with a ``double-loop zero-chain'' property (Lemma~\ref{lem: zero chain}), ensuring both that each individual block cannot progress too quickly (i.e., an ``inner chain''), and also that activating a block requires substantial progress in the preceding block (i.e., an ``outer chain'').
The main technical challenge in the analysis is further guaranteeing that convex combinations of gradients, which can activate arbitrary blocks, remain large. To this end, for every point $\bx$ with a zero final block,
we explicitly construct a unit vector $\bw'(\bx)\in\SS^{d-1}$ whose inner product with all nearby gradients is $\Omega(1/\sqrt{n})$, implying that any
convex combination of nearby gradients cannot be smaller than order $1/\sqrt{n}\approx\eps$ (Lemma~\ref{lem: goldstein large}).

Next, we extend our lower bound to the task of finding $(\lambda,\eps)$-stationary points:

\begin{theorem} \label{thm: lambda}
For some absolute constant $c>0$ the following holds:
For any $L,\Delta,\lambda,\eps>0$ satisfying $\eps\leq c\cdot\min\{L,\Delta^{2/3}\lambda^{1/3},L^4\Delta^{-2}\lambda^{-1}\}$, there exists an $L$-Lipschitz function $f:\reals^d\to\reals$, where $d=\Theta({\Delta L^2\sqrt{\lambda}}{\eps^{-7/2}})$,
satisfying $f(\bzero)-\inf f\leq\Delta$, such that for any zero-respecting first-order algorithm, all of its first $T$ iterates are not $(\lambda,\eps)$-stationary points of $f$, unless
\[
T\geq c\cdot \frac{\Delta L^2\sqrt{\lambda}}{\eps^{7/2}}~.
\]
\end{theorem}

As mentioned, Theorem~\ref{thm: lambda}
is tight up to a numerical constant due to existing upper bounds, thus establishing the optimal first-order oracle complexity of finding $(\lambda,\eps)$-stationary points. 
Once again, this rate is matched by stochastic algorithms \citep{zhang2024randomscaling}, 
therefore establishing that the first-order oracle complexity of finding stationary points under this relaxed notion does not depend on the stochasticity of the gradients as well.
We are not aware of any previous lower bound for this task when having access to deterministic gradients.

The proof of Theorem~\ref{thm: lambda} considers the same lower bound construction as in the proof of Theorem~\ref{thm: main}. The only missing ingredient is that we further show that all subgradients, even at points with non-zero last block, are non-negatively correlated with the aforementioned direction $\bw'$. Therefore, taking combinations with far-away gradients cannot help in 
decreasing $\Pcal_\lambda$, since assigning probability mass to distant points in order to minimize the combination's norm is compensated by the quadratic distance penalty.

\section{Proofs} \label{sec: proofs}

\subsection{Proof of Theorem~\ref{thm: main}}

Let $d=Bn$ for some sufficiently large natural number $B\geq 2$ and odd number $n\geq 99$ to be specified later.
Throughout the proof, given $\bx\in\reals^d$, we denote $\bx=(\bx^1,\dots,\bx^B)$ where $\bx^i\in\reals^n$, namely, we view a vector as consisting of $B$ ``blocks'' of $n$ coordinates each.

We start the proof by introducing several functions which will be used to construct our lower bound.
Denote the smoothstep function $s:\reals\to[0,1]$
and the odd smoothstep $\psi:\reals\to[-1,1]$ (see Figure~\ref{fig:s-psi}) as

\begin{center}
\begin{minipage}[c]{0.66\textwidth}
\vspace{-24pt}
\[
s(t):=\begin{cases}
    0, & t\leq0\\
    3t^2-2t^3, & t\in(0,1)\\
    1, & t\geq 1
\end{cases}
~,
\qquad
\psi(t):=s(t)-s(-t)~.
\]
\end{minipage}
\begin{minipage}[c]{0.32\textwidth}
\centering
\raisebox{-6pt}{
\includegraphics[height=7em]{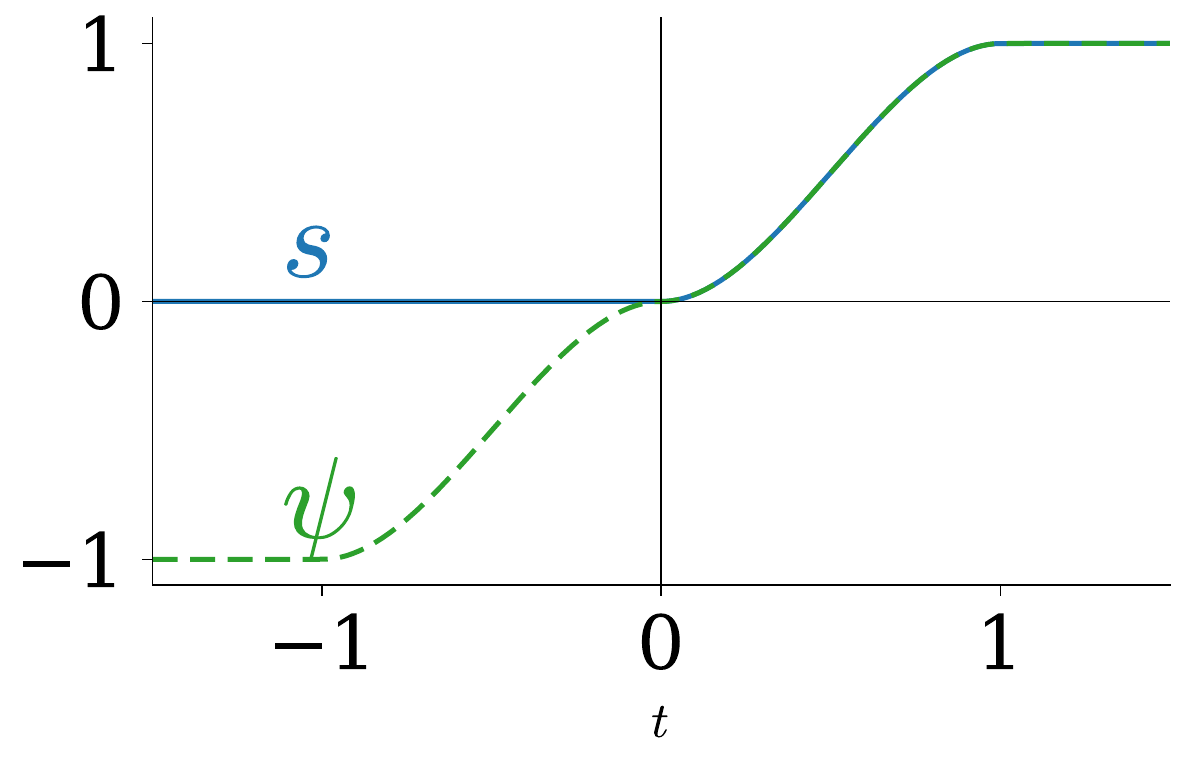}}
\vspace{-8pt}\captionof{figure}{plot of $s$ and $\psi$.}
    \label{fig:s-psi}
\end{minipage}
\end{center}
Throughout the proof we will use that fact that $0\leq s',\psi'\leq\frac{3}{2}$.
For a block $\bz\in\reals^n$, we define
\[
\bar{\psi}(\bz):=\frac{1}{n}\sum_{j\in[n]}\psi(z_j)~,
\qquad
\widetilde{\psi}(\bz):=\frac{1}{n}\sum_{j\in[n]}\psi(z_j)^2~,
\]
and also let
\begin{align*}
h(\bz)&:=\left[\mathrm{median}_{j\in[n]}\left(\theta_j-z_j\right)\right]_{[-1,1]}~,
\qquad \theta_j:=\frac{j-(n+1)/2}{10^4n\sqrt{B}}~,
\end{align*}
where $[a]_{[-1,1]}:=\min\{1,\,\max\{-1,\,a\}\}$ denotes the clipping to $[-1,1]$.

We consider the (unscaled) hard function $F:\reals^d\to\reals$ to be
\begin{align*}
F(\bx)&:=\left(\left[h(\bx^1)-2s\left(16\widetilde{\psi}(\bx^1)-1\right)+\frac{3}{4}\right]_{+}^2
+\sum_{i=2}^{B}\left[h(\bx^i)+s\left(32\widetilde{\psi}(\bx^{i-1})-1\right)-2s\left(16\widetilde{\psi}(\bx^i)-1\right)-\frac{1}{4}\right]_{+}^2
\right)^{1/2}
\\&~~~~~~~-100\sum_{i\in[B]}\bar{\psi}(\bx^i)
~.
\end{align*}
Equivalently, it will be useful to denote
\begin{align*}
 \sigma_1(\bx)&:=h(\bx^1)-2s\bigl(16\widetilde{\psi}(\bx^1)-1\bigr)+\frac{3}{4}~,\\
 \sigma_i(\bx)&:=h(\bx^i)+s\bigl(32\widetilde{\psi}(\bx^{i-1})-1\bigr)
 -2s\bigl(16\widetilde{\psi}(\bx^i)-1\bigr)-\frac{1}{4}~,
 \qquad i\in\{2,\dots,B\},
 \\\bar{\Psi}(\bx)&:=\sum_{i\in[B]}\bar{\psi}(\bx^i)~,
\end{align*}
and observe that
\begin{equation} \label{eq: F as sigma}
F(\bx)=\norm{[\bsigma(\bx)]_+}-100\bar{\Psi}(\bx)~,
\end{equation}
where the ReLU function $[\,\cdot\,]_+$ is applied coordinate-wise. Intuitively, the construction consists of three components: $(1)$ a shifted median $h$ ensuring that coordinates are revealed sequentially within each block; $(2)$ smoothstep terms $s$ ensuring that blocks are revealed sequentially; $(3)$ a term $-100\bar{\Psi}$ resulting a non-negligible gradient direction preventing Goldstein stationarity until all blocks are non-zero.

\begin{lemma}\label{lem: F Lip}
The function $F$ is $(49+150\sqrt{B/n})$-Lipschitz, 
and satisfies $F(\bzero)-\inf F\leq 200B$.
\end{lemma}

\begin{proof}[Proof of Lemma~\ref{lem: F Lip}]
It is easy to verify that $\bx\mapsto(h(\bx^i))_{i\in[B]}$ is $1$-Lipschitz, that $\bx\mapsto(\widetilde{\psi}(\bx^i))_{i\in[B]}$ is $(3/\sqrt{n})$-Lipschitz, thus also $1$-Lipschitz, and that each of the maps
$\bz\mapsto 2s\bigl(16\widetilde{\psi}(\bz)-1\bigr),~\bz\mapsto s\bigl(32\widetilde{\psi}(\bz)-1\bigr)$
has Lipschitz constant at most $48\cdot3/\sqrt{n}$. Hence $\bsigma:\reals^d\to\reals^B$ has Lipschitz constant bounded by $1+288/\sqrt{n}\leq49$, and therefore so does $\bx\mapsto \|[\bsigma(\bx)]_{+}\|$.
To bound the additional term in \eqref{eq: F as sigma}, a straightforward calculation shows that $|\psi'(t)|\leq\frac{3}{2}$, hence $\psi$ is $\frac{3}{2}$-Lipschitz, and therefore $\bar{\psi}$ is $\frac{3}{2\sqrt{n}}$-Lipschitz. By orthogonality of different blocks, we get that for any $\bx,\by\in\reals^d:$ 
\begin{align*}
\left|\sum_{i\in[B]}\bar{\psi}(\bx^i)-\sum_{i\in[B]}\bar{\psi}(\by^i)\right|
&\leq \sum_{i\in[B]}\left|\bar{\psi}(\bx^i)-\bar{\psi}(\by^i)\right|
\leq \frac{3}{2\sqrt{n}}\sum_{i\in[B]}\|\bx^i-\by^i\|
\leq  \frac{3}{2\sqrt{n}}\sqrt{B}\left(\sum_{i\in[B]}\|\bx^i-\by^i\|^2\right)^{1/2}
\\&=\frac{3}{2}\sqrt{\frac{B}{n}}\|\bx-\by\|~.
\end{align*}
Overall, we see that $F$ has Lipschitz constant of at most $49+100\cdot \frac{3}{2}\sqrt{B/n}$.

To bound the suboptimality at zero, first note that $h(\bzero)=\bar{\psi}(\bzero)=\widetilde{\psi}(\bzero)=0$, so $F(\bzero)=\frac{3}{4}$.
Further noting that $\psi\leq 1$ and so $\bar{\psi}\leq 1$,
by \eqref{eq: F as sigma} it is clear that $F\geq 0-100B$. Overall, $F(\bzero)-\inf F\leq \frac{3}{4}+100B<200B$.

\end{proof}

\begin{lemma} \label{lem: zero chain}
Let $\bx\in\reals^d$ and $i\in[B]$.
It holds that:
\begin{enumerate}[label=(\roman*)]
    \item If $\supp(\bx^i)\subseteq J\subsetneq[n]$, then there exists $j^*\notin J$ such that $\supp(\bg^i)\subseteq J\cup\{j^*\}$ for all $\bg\in\partial F(\bx)$.
    \item If $\bx^i=\bzero$ and $|\supp(\bx^{i-1})|\leq\frac{n}{32}$ for some $i\geq 2$, then $\bg^{i}=\bzero$ for all $\bg\in\partial F(\bx)$.
\end{enumerate}
Thus, if a zero-respecting algorithm is applied to $F$ for $T\leq\frac{d}{64}$ iterations, all of its iterates $\bx_1,\dots,\bx_T$ must satisfy $\bx_t^{B}=\bzero$.
\end{lemma}

\begin{proof}[Proof of Lemma~\ref{lem: zero chain}]

Note that by definition of $h$, for any $\bz\in\reals^n:$
\begin{equation} \label{eq: Gh form}
\partial h(\bz)\subseteq \Gcal(\bz):=\begin{cases}
    \conv\{-\be_j~:~h(\bz)=\theta_j-z_j\}, & |h(\bz)|<1\\
    \conv\left(\{-\be_j~:~h(\bz)=\theta_j-z_j\}\cup\{\bzero\}\right), & |h(\bz)|=1
\end{cases}
~.
\end{equation}
Further using the facts that $\partial (f_1+f_2)\subseteq\partial f_1+\partial f_2$ and $\partial(f_1\circ f_2)\subseteq\conv\{g_1\bg_2:g_1\in\partial f_1(f_2(\bx)),\bg_2\in\partial f_2(\bx)\}$ \citep[Proposition 2.3.3, Theorem 2.3.9]{Clarke-1990-Optimization}
we get via direct differentiation that any $\bg\in\partial F(\bx)$ has the block form
\begin{align}
&\bg^i=w_i(\bx) \bu^i+32\left[w_{i+1}(\bx)s'\bigl(32\widetilde{\psi}(\bx^i)-1\bigr)
-w_i(\bx) s'\bigl(16\widetilde{\psi}(\bx^i)-1\bigr)\right]
\nabla\widetilde{\psi}(\bx^i)
-100\nabla\bar{\psi}(\bx^i)
\label{eq: g form}
\\
&\text{where}~~
\bu^i\in\Gcal(\bx^i),
\qquad
\bw(\bx)\in
\partial\Phi(\bsigma(\bx))\subseteq\{\ba\in[0,1]^B:\|\ba\|\leq1\},
\qquad \Phi(\ba):=\|[\ba]_+\|
~.\nonumber
\end{align}
Here $w_{B+1}(\bx):=0$, and $w_i(\bx)=0$ whenever $\sigma_i(\bx)<0$; if $[\bsigma(\bx)]_+\neq\bzero$, then $\bw(\bx)=[\bsigma(\bx)]_+/\|[\bsigma(\bx)]_+\|$.

We claim that for all $\bx\in\reals^d,~i\in[B],~j\in[n]:$
\begin{align} \label{eq: gij=0 condition}
x^i_j=0\text{~~and~~}\left(\sigma_i(\bx)<0\text{~~or~~}\theta_j-x^i_j\neq h(\bx^i)\right)
~~~\implies~~~ \forall\bg\in\partial F(\bx):~g^i_j=0 ~.
\end{align}
Indeed, looking at \eqref{eq: g form} and noting that $\psi'(0)=0$, we see that the second and third terms which depend on $\nabla\widetilde{\psi},\nabla\bar{\psi}$ vanish when $x^i_j=0$. For the remaining first term $w_i(\bx)\bu^i$, it holds that either $\sigma_i(\bx)<0$ which implies $w_i(\bx)=0$, or else $\theta_j-x^i_j\neq h(\bx^i)$ and so $u^i_j=0$
for all $\bu^i\in\Gcal(\bx^i)$
since this coordinate is not active according to \eqref{eq: Gh form}. In either case, $w_i(\bx)u^i_j=0$ thus proving \eqref{eq: gij=0 condition}.

To prove $(i)$, fix $i\in[B]$ and suppose $\supp(\bx^i)\subseteq J\subsetneq[n]$. Since for any $j\notin J$ it holds that $x^i_j=0$, and by noting that $(\theta_j)_{j\in[n]}\subset(-1,1)$ are distinct, we see that there is at most one $j\notin J$ attaining $h(\bx^i)=\theta_{j}-x^i_{j}$. Let $j^*\notin J$ be this index (if it is not attained by any such index, pick arbitrarily).
We get that for any $j\notin J\cup\{j^*\}$ both $x^i_j=0$ and $\theta_j-x^i_j\neq h(\bx^i)$, hence $g^i_j= 0$ for all $\bg\in\partial F(\bx)$ by \eqref{eq: gij=0 condition}, proving that $\supp(\bg^i)\subseteq J\cup\{j^*\}$ as claimed.

To prove $(ii)$, suppose $\bx^i=\bzero$ and $|\supp(\bx^{i-1})|\leq\frac{n}{32}$. Since $\psi(0)=0$ and $\psi^2\leq 1$, it holds that
\begin{align*}
&\widetilde{\psi}(\bx^{i-1})=\frac{1}{n}\sum_{j\in[n]}\psi(x^{i-1}_j)^2\leq \frac{|\supp(\bx^{i-1})|}{n}\leq \frac{1}{32}
\\ \implies~~ & 32\widetilde{\psi}(\bx^{i-1})-1\leq 0
\\ \implies~~ & s(32\widetilde{\psi}(\bx^{i-1})-1)=0~.
\end{align*}
Further note that $h(\bx^i)=\widetilde{\psi}(\bx^i)=0$, so overall
\[
\sigma_i(\bx)=\underset{=0}{\underbrace{h(\bx^i)}}+\underset{=0}{\underbrace{s\bigl(32\widetilde{\psi}(\bx^{i-1})-1\bigr)}}
 -\underset{=0}{\underbrace{2s\bigl(16\widetilde{\psi}(\bx^i)-1\bigr)}}-\frac{1}{4}
<0
 ~.
\]
Thus, by \eqref{eq: gij=0 condition} we get that $\bg^i=\bzero$ for all $\bg\in\partial F(\bx)$ as claimed.

We now prove the final conclusion.
By the zero-respecting property, item $(i)$, applied to the cumulative set of coordinates revealed in a block, shows that this set can grow by at most one index at each iteration. Hence, for the support within a block to grow beyond $\frac{n}{32}$, any zero-respecting algorithm requires at least $\lfloor n/32\rfloor+1$ iterations after that block first becomes accessible. Moreover, item $(ii)$ further shows that block $i$ cannot first become accessible unless $|\supp(\bx^{i-1})|>\frac{n}{32}$. Therefore, to get to a point $\bx$ such that its final block $\bx^B\neq\bzero$, any zero-respecting algorithm requires a number of iterations which is at least $(B-1)(\lfloor n/32\rfloor+1)>(B-1)\frac{n}{32}\geq\frac{Bn}{64}$, where the last inequality uses $B\geq2$.

\end{proof}

\begin{lemma} \label{lem: g<0}
    For any $\bx\in\reals^d,~\bg\in\partial F(\bx):~\bg\vleq-4\nabla\bar{\Psi}(\bx) \vleq\bzero$.
\end{lemma}

\begin{proof}[Proof of Lemma~\ref{lem: g<0}]
The fact that $-4\nabla\bar{\Psi}(\bx)\vleq \bzero$ is easily seen by recalling that $\bar{\Psi}$ increases in each coordinate, and so it remains to prove the inequality $\bg\vleq -4\nabla\bar{\Psi}(\bx)$.
Let $\bg\in\partial F(\bx)$ and recall its block form given by \eqref{eq: g form}. Denote $a_i:=32[w_{i+1}(\bx)s'\bigl(32\widetilde{\psi}(\bx^i)-1\bigr)
-w_i(\bx) s'\bigl(16\widetilde{\psi}(\bx^i)-1\bigr)]$, i.e., the coefficient of $\nabla\widetilde{\psi}(\bx^i)$.
Since $0\leq s'\leq\frac{3}{2}$ it holds that $|a_i|\leq 48$. Further observe that $\bu^i\vleq 0$ by \eqref{eq: Gh form}, and $\bw(\bx)\vgeq0$, so $w_i(\bx)\bu^i\vleq \bzero$.
Lastly, since $|\psi|\leq 1$ we get that $|\nabla\widetilde{\psi}|\vleq 2\nabla\bar{\psi}$. Overall, plugging into \eqref{eq: g form} we obtain
\begin{equation} \label{eq: g<-4}
\bg^i~\vleq~ \bzero+48\nabla\widetilde{\psi}(\bx^i)-100\nabla\bar{\psi}(\bx^i)
~\vleq~ 48|\nabla\widetilde{\psi}(\bx^i)|-100\nabla\bar{\psi}(\bx^i)
~\vleq~ -4\nabla\bar{\psi}(\bx^i)~,
\end{equation}
hence $\bg\vleq -4\nabla\bar{\Psi}(\bx)$.

\end{proof}

\begin{lemma} \label{lem: gi1>=}
    For any $\bx\in\reals^d,~\bg\in\partial F(\bx),~i\in[B]$, if $[\bsigma(\bx)]_+\neq \bzero$ then
    $-\sum_{j\in[n]}g^i_j\geq \frac{[\sigma_i(\bx)]_+}{\|[\bsigma(\bx)]_+\|}$ .
\end{lemma}

\begin{proof}[Proof of Lemma~\ref{lem: gi1>=}]

First, we claim that for any $\bz\in\reals^n:$
\begin{equation} \label{eq: tildepsi<1/8}
\widetilde{\psi}(\bz)<\frac{1}{8}
~~~\implies~~~
|h(\bz)|<1\text{~~and~~}
\|\nabla\bar{\psi}(\bz)\|\geq \frac{[|h(\bz)|-\frac{1}{10^4\sqrt{B}}]_+}{2\sqrt{n}}~.
\end{equation}
Indeed, assuming $\widetilde{\psi}(\bz)=\frac{1}{n}\sum_{j\in[n]}\psi(z_j)^2<\frac{1}{8}$, since $\psi(\frac{3}{4})^2>\half$, at most $n/4$ coordinates of $\bz$ can satisfy $|z_j|>\frac{3}{4}$. Noting that $|\theta_j|:=|\frac{j-(n+1)/2}{10^4n\sqrt{B}}|\leq \frac{1}{10^4\sqrt{B}}$, we see that the median $|h(\bz)|<1$, proving the first claim.
Moreover, by definition of the median, at least half of the coordinates satisfy $|z_j|\geq|h(\bz)|-\frac{1}{10^4\sqrt{B}}$, and at least $n/4$ of these coordinates satisfy $|z_j|\leq\frac{3}{4}$. Denote a set of such coordinates by $J$, so $|J|\geq n/4$.
Noting that by direct calculation $\psi'(z_j)\geq |z_j|$ on $[-\frac{3}{4},\frac{3}{4}]$, since $[\nabla\bar{\psi}(\bz)]_j=\frac{\psi'(z_j)}{n}$ we get
\[
\|\nabla\bar{\psi}(\bz)\|
=\sqrt{\sum_{j\in[n]}\frac{\psi'(z_j)^2}{n^2}}
\geq\frac{1}{n}\sqrt{\sum_{j\in J}|z_j|^2}
\geq\frac{\sqrt{|J|}}{n}\left[|h(\bz)|-\frac{1}{10^4\sqrt{B}}\right]_+
\geq\frac{[|h(\bz)|-\frac{1}{10^4\sqrt{B}}]_+}{2\sqrt{n}}~,
\]
proving \eqref{eq: tildepsi<1/8}.

Now, let $\bx\in\reals^d$ such that $[\bsigma(\bx)]_+\neq\bzero$, and let $i\in[B]$. If $[\sigma_i(\bx)]_+=0$, then the lemma immediately follows from Lemma~\ref{lem: g<0}. Otherwise, assume that $[\sigma_i(\bx)]_+>0$. Note that this implies that $\widetilde{\psi}(\bx^i)<\frac{1}{8}$, since otherwise $-s(16\widetilde{\psi}(\bx^i)-1)\leq-1$ and by definition of $\sigma_i(\bx)$ it could not have been positive. Therefore, by \eqref{eq: tildepsi<1/8} it holds that $|h(\bx^i)|<1$, which further implies via \eqref{eq: Gh form} that any $\bu^i\in\Gcal(\bx^i)$ satisfies $\sum_{j\in[n]}u^i_j=-1$.
Recalling the block form of $\bg\in\partial F(\bx)$ given by \eqref{eq: g form}, and noting that
$w_i(\bx)=\frac{[\sigma_i(\bx)]_+}{\|[\bsigma(\bx)]_+\|}$ since $[\sigma_i(\bx)]_+\neq0$, and the non-positivity of the terms established in \eqref{eq: g<-4}, we overall get
\[
\sum_{j\in[n]}g^i_j
\leq w_i(\bx)\sum_{j\in[n]}u^i_j
+0
=-w_i(\bx)
=-\frac{[\sigma_i(\bx)]_+}{\|[\bsigma(\bx)]_+\|}~.
\]

\end{proof}

\begin{lemma} \label{lem: r function}
Define $\br:\reals^d\to\reals^B$ as $r_i(\bx):=[\sigma_i(\bx)-h(\bx^i)]_+$ for all $i\in[B]$. Then for any $\bx,\by\in\reals^d:$
\[
\|[\bsigma(\bx)]_+-\br(\by)\|
\leq
2\sqrt{n}\left\|\nabla\bar{\Psi}(\bx)\right\|+\frac{1}{10^4}+\frac{288\|\bx-\by\|}{\sqrt{n}}
~.
\]

\end{lemma}

\begin{proof}
For any $i\in[B]$, the scalar ReLU is $1$-Lipschitz, so
\[
|[\sigma_i(\bx)]_+-r_i(\bx)|\leq |h(\bx^i)|~.
\]
If $\widetilde{\psi}(\bx^i)\geq1/8$, then $2s(16\widetilde{\psi}(\bx^i)-1)=2$, and both $[\sigma_i(\bx)]_+$ and $r_i(\bx)$ vanish. Otherwise, \eqref{eq: tildepsi<1/8} gives
\[
|[\sigma_i(\bx)]_+-r_i(\bx)|
\leq 2\sqrt{n}\|\nabla\bar{\psi}(\bx^i)\|+\frac{1}{10^4\sqrt{B}}~.
\]
Consequently, by the triangle inequality in $\reals^B$ and orthogonality of different blocks,
\[
\|[\bsigma(\bx)]_+-\br(\bx)\|
\leq2\sqrt{n}\|\nabla\bar{\Psi}(\bx)\|+\frac{1}{10^4}~.
\]
The map $\bx\mapsto(\widetilde{\psi}(\bx^i))_{i\in[B]}$ is $3/\sqrt{n}$-Lipschitz, and both scalar functions $x\mapsto s(32x-1)$ and $x\mapsto 2s(16x-1)$ are $48$-Lipschitz, hence $\br$ is $288/\sqrt{n}$-Lipschitz. For $\by\in\reals^d$, another triangle inequality yields
\[
\|[\bsigma(\bx)]_+-\br(\by)\|
\leq2\sqrt{n}\|\nabla\bar{\Psi}(\bx)\|+\frac{1}{10^4}+\frac{288\|\bx-\by\|}{\sqrt{n}}~.
\]
\end{proof}

\begin{lemma} \label{lem: goldstein large}
Let $\delta'\leq\frac{\sqrt{n}}{5000}$.
For any $\bx\in\reals^d$ satisfying $\bx^B=\bzero$, there is a unit vector $\bw'(\bx)\vleq\bzero$ such that for all $\by\in\BB(\bx,\delta'),~\bg\in\partial F(\by):~\inner{\bw'(\bx),\bg}\geq\frac{1}{4\sqrt{n}}$. In particular, $\bx$ is not a $(\delta',\frac{1}{5\sqrt{n}})$-Goldstein stationarity point of $F$.
\end{lemma}

\begin{proof}[Proof of Lemma~\ref{lem: goldstein large}]

Fix $\bx\in\reals^d$ such that $\bx^B=\bzero$.
Let $\by\in\BB(\bx,\delta')$ and $\bg\in\partial F(\by)$.
Fix $i\in[B]$ to be the minimal index such that $\widetilde{\psi}(\bx^{i})\leq \frac{1}{16}$ (such an $i$ exists since $\widetilde{\psi}(\bx^{B})=\widetilde{\psi}(\bzero)=0$). If $i=1$, then $r_1(\bx)=3/4$ directly. If $i\geq2$, note that
\begin{align*}
r_i(\bx)
=[\sigma_i(\bx)-h(\bx^i)]_+
=\Biggl[s\bigl(\underset{>1}{\underbrace{32\widetilde{\psi}(\bx^{i-1})-1}}\bigr)
 -2s\bigl(\underset{\leq 0}{\underbrace{16\widetilde{\psi}(\bx^i)-1}}\bigr)-\frac{1}{4}\Biggr]_+
 =\left[1-\frac{1}{4}\right]_+=\frac{3}{4}~,
\end{align*}
and so
\begin{align} \label{eq: r>3/4}
\|\br(\bx)\|
\geq \left|r_i(\bx)\right|\geq \frac{3}{4}
> \half
~.
\end{align}
Define $\bv\in\reals^d$ in block form as
\[
\bv^i:=-\frac{r_i(\bx)}{\sqrt{n}\|\br(\bx)\|}\bone~,
\]
and further let
\[
\bu:=\begin{cases}
-\frac{\nabla\bar{\Psi}(\bx)}{\|\nabla\bar{\Psi}(\bx)\|}, & \nabla\bar{\Psi}(\bx)\neq\bzero\\
\bzero, & \text{otherwise}
\end{cases}~.
\]
If $\bu\neq\bzero$, then we see that
\begin{align}
\inner{\bu,\bg}
&=\binner{\frac{\nabla\bar{\Psi}(\bx)}{\|\nabla\bar{\Psi}(\bx)\|},-\bg}
\overset{\text{Lemma~\ref{lem: g<0}}}{\geq}4\binner{\frac{\nabla\bar{\Psi}(\bx)}{\|\nabla\bar{\Psi}(\bx)\|},\nabla\bar{\Psi}(\by)}
=4\|\nabla\bar{\Psi}(\bx)\|+4\binner{\frac{\nabla\bar{\Psi}(\bx)}{\|\nabla\bar{\Psi}(\bx)\|},\nabla\bar{\Psi}(\by)-\nabla\bar{\Psi}(\bx)}
\nonumber\\&\geq4\|\nabla\bar{\Psi}(\bx)\|-4\|\nabla\bar{\Psi}(\by)-\nabla\bar{\Psi}(\bx)\|
\nonumber\\&\geq 4\|\nabla\bar{\Psi}(\bx)\|-\frac{24\delta'}{n}
\label{eq: <ug>}
\end{align}
where the last inequality holds since $\nabla\bar{\Psi}$ is $\frac{6}{n}$-Lipschitz, which is easily verified by direct calculation.
The inequality above clearly holds if $\bu=\bzero$ as well.
Applying Lemma~\ref{lem: r function} and again using the fact that $\nabla\bar{\Psi}$ is $6/n$-Lipschitz gives
\begin{align}
\|[\bsigma(\by)]_+-\br(\bx)\|
\leq2\sqrt{n}\|\nabla\bar{\Psi}(\by)\|+\frac{1}{10^4}+\frac{288\delta'}{\sqrt{n}}
\leq2\sqrt{n}\|\nabla\bar{\Psi}(\bx)\|+\frac{1}{10^4}+\frac{300\delta'}{\sqrt{n}}~.
\label{eq: nearby sigma-r}
\end{align}
Moreover, if $[\bsigma(\by)]_+\neq \bzero$ then
\begin{align}
\inner{\bv,\bg}
&= \sum_{i\in[B]}\inner{\bv^i,\bg^i}
= \frac{1}{\sqrt{n}\|\br(\bx)\|}\sum_{i\in[B]}r_i(\bx)\left(-\sum_{j\in[n]}g^i_j\right)
\nonumber\\&\geq \frac{1}{\sqrt{n}\|\br(\bx)\|}\sum_{i\in[B]}r_i(\bx)\frac{[\sigma_i(\by)]_+}{\|[\bsigma(\by)]_+\|}
= 
\frac{1}{\sqrt{n}}\binner{\frac{\br(\bx)}{\|\br(\bx)\|},\frac{[\bsigma(\by)]_+}{\|[\bsigma(\by)]_+\|}}
\nonumber\\&\geq 
\frac{1}{\sqrt{n}}\left(1-\frac{2\|\br(\bx)-[\bsigma(\by)]_+\|}{\|\br(\bx)\|}\right)
\nonumber\\&\geq \frac{1}{\sqrt{n}}\left(1-4\|\br(\bx)-[\bsigma(\by)]_+\|\right)
\nonumber\\
&\geq\frac{1}{\sqrt{n}}-8\left\|\nabla\bar{\Psi}(\bx)\right\|-\frac{4}{10^4\sqrt{n}}-\frac{1200\delta'}{n}
\label{eq: <vg>}
\end{align}
where the first inequality is due to Lemma~\ref{lem: gi1>=}, the second is the elementary inequality $\inner{\frac{\ba}{\|\ba\|},\frac{\bb}{\|\bb\|}}\geq 1-\frac{2\|\ba-\bb\|}{\|\ba\|}$ for any $\ba,\bb\neq\bzero$, the third is by \eqref{eq: r>3/4}, and the last is by \eqref{eq: nearby sigma-r}.
The inequality above also holds if $[\bsigma(\by)]_+=\bzero$, since in that case $\inner{\bv,\bg}\geq 0\geq \frac{1}{\sqrt{n}}\left(1-4\|\br(\bx)\|\right)$ by Lemma~\ref{lem: g<0} and \eqref{eq: r>3/4}, and \eqref{eq: nearby sigma-r} again gives the same final lower bound.

Let $\bw':=\frac{2\bu+\bv}{\|2\bu+\bv\|}$.
Combining Eqs. (\ref{eq: <ug>}), (\ref{eq: <vg>}) and the fact that by construction $\|\bu\|\leq1$, $\|\bv\|=1$, and $\bu,\bv\vleq\bzero$, thus $0<\|2\bu+\bv\|\leq 3$,
we get that
\begin{align*}
\binner{\bw',\bg}
\geq
\frac{1}{3}\left(8\|\nabla\bar{\Psi}(\bx)\|-\frac{48\delta'}{n}
+\frac{1}{\sqrt{n}}-8\left\|\nabla\bar{\Psi}(\bx)\right\|-\frac{4}{10^4\sqrt{n}}-\frac{1200\delta'}{n}\right)
\geq \frac{1}{4\sqrt{n}}~,
\end{align*}
the last inequality following from the assumption on $\delta'$.

Finally, since the inequality above holds for any $\bg\in\partial F(\by),~\by\in\BB(\bx,\delta')$,
we see that it must also hold for any $\bg\in\partial_{\delta'} F(\bx)$ by taking convex combinations. Thus, for any $\bg\in\partial_{\delta'} F(\bx):$
\[
\dist(\bzero,\partial_{\delta'} F(\bx))
\geq\|\bg\|
\geq\binner{\bw',\bg}
>\frac{1}{5\sqrt{n}}~.
\]

\end{proof}

We are ready to complete the proof.
Given $L,\Delta,\delta,\eps>0$ as in the theorem, set
\[
N:=\left(\frac{L}{1600\eps}\right)^2,\qquad
n:=2\left\lfloor\frac{N-1}{2}\right\rfloor+1,\qquad
B:=\left\lfloor\frac{\Delta}{8\cdot10^6\delta\eps}\right\rfloor,
\qquad d:=Bn~.
\]
Then $n\geq99$ is odd, $N/2\leq n\leq N$, and $2\leq B\leq n$.
Consider the scaled function
\[
f(\bx):=\frac{\Delta}{200 B}F\left(\frac{1600B\sqrt{n}\eps}{\Delta}\cdot \bx\right)~.
\]
By Lemma~\ref{lem: F Lip} it holds that
\[
f(\bzero)-\inf f
\leq\frac{\Delta}{200B}\cdot 200B
\leq\Delta
\]
and, since $B\leq n$,
\[
\mathrm{Lip}(f)\leq \frac{\Delta}{200B}\cdot\frac{1600B\sqrt{n}\eps}{\Delta}\cdot 200
=1600\sqrt{n}\eps\leq L~.
\]
Furthermore, the radius after scaling satisfies
\[
\delta':=\frac{1600B\sqrt{n}\eps}{\Delta}\cdot\delta
\leq\frac{\sqrt{n}}{5000}~.
\]
Since scaling does not affect the gradient support, it still holds by Lemma~\ref{lem: zero chain} that when applying a zero-respecting algorithm to $f$, all the iterates $\bx_1,\dots,\bx_T$ must satisfy $\bx_t^{B}=\bzero$ unless
\[
T> \frac{Bn}{64}
=\Theta\left(\frac{\Delta L^2}{\delta\eps^3}\right)
~.
\]
But whenever $\bx_t^{B}=\bzero$ it also holds that
$\frac{1600B\sqrt{n}\eps}{\Delta}
\cdot\bx_t^B=\bzero$, hence Lemma~\ref{lem: goldstein large} ensures that
\[
\dist(\bzero,\partial_{\delta}f(\bx_t))
=8\sqrt{n}\eps\cdot\dist\left(\bzero,\partial_{\delta'}F\left(\frac{1600B\sqrt{n}\eps}{\Delta}
\cdot\bx_t\right)\right)
\geq \frac{8\sqrt{n}\eps}{5\sqrt{n}}=\frac{8}{5}\eps>\eps~.
\]

\subsection{Proof of Theorem~\ref{thm: lambda}}

We use the same unscaled hard function $F$ and block notation as in the proof of Theorem~\ref{thm: main}.
Given $L,\Delta,\lambda,\eps>0$ as in the theorem, let
\[
\delta:=2\sqrt{\frac{\eps}{\lambda}}~,
\]
and set
\[
N:=\left(\frac{L}{1600\eps}\right)^2,\qquad
n:=2\left\lfloor\frac{N-1}{2}\right\rfloor+1,\qquad
B:=\left\lfloor\frac{\Delta\sqrt{\lambda}}{16\cdot10^6\eps^{3/2}}\right\rfloor,
\qquad d:=Bn~.
\]
The assumptions in the theorem imply that $\delta\eps\leq\Delta/(16\cdot10^6)$, and the parameter assignment above are precisely those in the proof of Theorem~\ref{thm: main} for this choice of $\delta$.
In particular, $n\geq99$ is odd, $N/2\leq n\leq N$, and $2\leq B\leq n$.
Consider the same scaled function
\[
f(\bx):=\alpha F(\beta\bx),\qquad
\alpha:=\frac{\Delta}{200B},\qquad
\beta:=\frac{1600B\sqrt{n}\eps}{\Delta}~.
\]
As already verified using Lemma~\ref{lem: F Lip}, $f$ is $L$-Lipschitz and satisfies $f(\bzero)-\inf f\leq\Delta$.

We claim that whenever $\bx^B=\bzero$, it holds that $\Pcal_\lambda f(\bx)\geq2\eps$.
To see this, recall that by Lemma~\ref{lem: goldstein large}, applied at $\beta\bx$, there is a unit vector $\bw'\vleq\bzero$
satisfying for all $\bz\in\BB(\beta\bx,\beta\delta)\subseteq\BB(\beta\bx,\frac{\sqrt{n}}{5000}),~\bg'\in\partial F(\bz):~\inner{\bw',\bg'}\geq\frac{1}{4\sqrt{n}}$.
Consequently, since $\partial f(\by)=\alpha\beta\,\partial F(\beta\by)$, it holds that for all $\by\in\BB(\bx,\delta),~\bg\in\partial f(\by):$
\begin{equation}\label{eq: wg>2eps}
\inner{\bw',\bg}\geq\frac{\alpha\beta}{4\sqrt{n}}=2\eps~.
\end{equation}
Furthermore, Lemma~\ref{lem: g<0} implies that $\bg\vleq\bzero$ for every $\bg\in\partial f(\by),~\by\in\reals^d$, and since $\bw'\vleq\bzero$, we also have (without assuming $\by\in\BB(\bx,\delta)$):
\begin{equation}\label{eq: wg>=0}
\inner{\bw',\bg}\geq0~.
\end{equation}

Now, let $\by,\bg\in\reals^d$ be random vectors with $\bg\in\partial f(\by)$ almost surely.
If $\by\in\BB(\bx,\delta)$, then by \eqref{eq: wg>2eps}:
\[
\inner{\bw',\bg}+\lambda\|\bx-\by\|^2
\geq 2\eps+0
=2\eps~,
\]
while if $\by\notin \BB(\bx,\delta)$, then by \eqref{eq: wg>=0}:
\[
\inner{\bw',\bg}+\lambda\|\bx-\by\|^2
\geq 0+\lambda\delta^2=4\eps~,
\]
and so overall
\[
\|\E[\bg]\|+\lambda\E\|\bx-\by\|^2
\geq
\E\left[\inner{\bw',\bg}+\lambda\E\|\bx-\by\|^2\right]
\geq 2\eps~.
\]
Taking the infimum over all such random vectors proves that $\Pcal_\lambda f(\bx)\geq2\eps$, as claimed.

Finally, positive scaling preserves subgradient supports, so Lemma~\ref{lem: zero chain} ensures that all of the first $T\leq Bn/64$ iterates of any zero-respecting algorithm satisfy $\bx_t^B=\bzero$.
For all these iterates, $\Pcal_\lambda f(\bx_t)\geq2\eps>\eps$, so they are not $(\lambda,\eps)$-stationary.
The parameter assignment gives
\[
\frac{Bn}{64}
=\Theta\left(\frac{\Delta L^2\sqrt{\lambda}}{\eps^{7/2}}\right)
~,
\]
completing the proof.

\section{Discussion}

In this work, we proved tight lower bounds which establish the first-order complexity of finding Goldstein stationary points, as well as a related relaxation thereof.
In particular, our results indicate that in nonsmooth nonconvex optimization, algorithms do not suffer from gradient stochasticity, thus answering an open question posed by \citet{cutkosky2023optimal}.

It is interesting to note that in \emph{smooth} nonconvex optimization, optimal convergence rates are in fact affected by gradient stochasticity ($\eps^{-2}$ vs. $\eps^{-4}$). By contrast, in nonsmooth \emph{convex} optimization, it is well-known that optimal complexities remain the same, since deterministic and stochastic subgradient methods converge similarly over convex functions (see e.g., \citealp{bubeck2015convex}).
Thus, it can be (very) informally argued that nonsmooth nonconvex optimization behaves more similarly to nonsmooth convex optimization, than to smooth nonconvex optimization.

Finally, we note that while we established tight lower bounds for (possibly randomized) zero-respecting algorithms, which already covers algorithms of interest, our lower bound construction can be extended via known reductions to account for arbitrary randomized algorithms which use a local oracle \citep{carmon2020lower}.
We plan this extension for a future version of this work.

\subsection*{Statement on AI use}

The author used ChatGPT 6 Pro when working on this paper.
In more detail,
the author had sought this result for several years ever since working on \citep{kornowski2022complexity}, in which a weaker lower bound was proved.
Throughout these years, the author made several failed attempts at proving the tight lower bound (including prompting various LLMs, which were useless for this problem). These endeavors resulted in an unpublished note based on the author's accumulated knowledge of the problem, in which partial progress was made based on the partition into coordinate blocks, as conveyed in the paragraphs following Theorem~\ref{thm: main}.
Recently, the author read \citet{li2026optimal}, in which the authors used GPT 5.6 Sol to prove an $\Omega(1/\eps^{4})$ lower bound on the deterministic complexity of weakly-convex optimization, using a very similar block-based approach.\footnote{Note that this setting is incomparable to ours: On the one hand, weakly-convex optimization is easier by imposing an additional assumption on the objective; however, the optimality criterion in that setting, of finding an approximate-stationary point of the Moreau envelope, implies being close to an $\eps$-stationary point, which is a stronger condition than Goldstein stationarity (and is, in fact, impossible to achieve efficiently in Lipschitz optimization, \citealp{kornowski2022oracle}).}
This led the author to revisit the problem using the latest version of ChatGPT, prompted with the aforementioned note attached. At first, it still failed. Yet, after a conversation in which both the author and the AI contributed ideas, the key missing idea of defining $h$ as a median (instead of a max, as is more typical to such lower-bound constructions) emerged.
The author then wrote the complete analysis, as well as the rest of the paper.
The author takes full responsibility for the paper's content.


\bibliographystyle{plainnat}
\bibliography{bib}

\end{document}

%% file: symbols.tex
\newtheorem{theorem}{Theorem}[section]

\newtheorem*{question*}{Question}

\newtheorem*{bigquestion*}{Big Question}

\newtheorem{lemma}[theorem]{Lemma}

\newtheorem{definition}[theorem]{Definition}
\newtheorem*{definition*}{Definition}
\newtheorem{remark}[theorem]{Remark}

\newtheorem*{keyquestion*}{Key Question}

\renewcommand{\eqref}[1]{Eq.~(\ref{#1})}

\newcommand{\reals}{\mathbb{R}}

\newcommand{\E}{\mathbb{E}}

\newcommand{\half}{\frac{1}{2}}

\newcommand{\bzero}{\mathbf{0}}
\newcommand{\bone}{\mathbf{1}}
\newcommand{\dist}{\mathrm{dist}}
\renewcommand{\SS}{\mathbb{S}}
\newcommand{\NN}{\mathbb{N}}

\newcommand{\ba}{\mathbf{a}}
\newcommand{\be}{\mathbf{e}}
\newcommand{\e}{\mathbf{e}}
\newcommand{\bx}{\mathbf{x}}
\newcommand{\x}{\mathbf{x}}
\newcommand{\bw}{\mathbf{w}}

\newcommand{\bg}{\mathbf{g}}
\newcommand{\g}{\mathbf{g}}
\newcommand{\bb}{\mathbf{b}}
\newcommand{\bu}{\mathbf{u}}

\newcommand{\bv}{\mathbf{v}}

\newcommand{\bz}{\mathbf{z}}
\newcommand{\br}{\mathbf{r}}

\newcommand{\by}{\mathbf{y}}
\newcommand{\y}{\mathbf{y}}

\newcommand{\bsigma}{\boldsymbol{\sigma}}

\newcommand{\Gcal}{\mathcal{G}}

\newcommand{\Pcal}{\mathcal{P}}

\newcommand{\Xcal}{\mathcal{X}}

\newcommand{\norm}[1]{\|#1\|}

\newcommand{\inner}[1]{\langle#1\rangle}
\newcommand{\binner}[1]{\left\langle#1\right\rangle}

\newcommand{\B}{\mathbb{B}}
\newcommand{\BB}{\mathbb{B}}

\newcommand{\conv}{\mathrm{conv}}

\renewcommand{\S}{\mathbb{S}}

\newcommand{\vleq}{\preceq}
\newcommand{\vgeq}{\succeq}

\newcommand{\supp}{\mathrm{supp}}

\newcommand{\eps}{\epsilon}